\documentclass[12pt]{amsart}
\usepackage{amsmath,amssymb,amscd}

\input xy
\xyoption{all}
\xyoption{arc}
\SelectTips{cm}{10}

\newtheorem{theorem}{Theorem}[section]

\newtheorem{lemma}[theorem]{Lemma}
\newtheorem{proposition}[theorem]{Proposition}

\theoremstyle{definition}

\newtheorem{remark}[theorem]{Remark}

\DeclareMathOperator{\rk}{\mathrm{rk}}
\DeclareMathOperator{\Det}{\mathsf{Det}}

\newcommand{\id}{\mathrm{id}}
\newcommand{\NA}{\widetilde{A}} %% normalisation of A
\newcommand{\CI}{\mathfrak{c}}  %% conductor ideal
\newcommand{\type}{\mathrm{t}}  %% type
\newcommand{\sing}{\mathsf{sing}}

\DeclareMathOperator{\ev}{\mathrm{ev}}
\DeclareMathOperator{\tors}{\mathsf{tors}}
\DeclareMathOperator{\Ann}{\mathsf{Ann}}
\DeclareMathOperator{\Coh}{\mathsf{Coh}}

\DeclareMathOperator{\Pic}{\mathsf{Pic}}
\DeclareMathOperator{\Mf}{\mathsf{M}}
\DeclareMathOperator{\Hom}{\mathsf{Hom}}
\DeclareMathOperator{\ShHom}{\mathcal{H}\mathsf{om}}
\DeclareMathOperator{\Ext}{\mathsf{Ext}}
\DeclareMathOperator{\ShExt}{\mathcal{E}\mathsf{xt}}

\DeclareMathOperator{\Ans}{\mathsf{An}}
\DeclareMathOperator{\Sets}{\mathsf{Sets}}

\newcommand{\ARG}{\,\cdot\,}

\newcommand{\kk}{\mathbb{C}}%% we work over \CC only
\newcommand{\CC}{\mathbb{C}}

\newcommand{\kA}{\mathcal{A}}
\newcommand{\kB}{\mathcal{B}}
\newcommand{\kC}{\mathcal{C}}
\newcommand{\kF}{\mathcal{F}}
\newcommand{\kG}{\mathcal{G}}

\newcommand{\kO}{\mathcal{O}}
\newcommand{\kL}{\mathcal{L}}
\newcommand{\kP}{\mathcal{P}}

\newcommand{\kT}{\mathcal{T}}

\newcommand{\lar}{\longrightarrow}

\begin{document}

\title[Analytic Moduli]{%
Analytic Moduli Spaces of Simple Sheaves on Families of Integral Curves}

\author{Igor Burban}
\address{%
Universit\"at zu K\"oln,
Mathematisches Institut,
Weyertal 86--90,
D--50931 K\"oln,
Germany
}
\email{burban@math.uni-koeln.de}

\author{Bernd Kreussler}
\address{%
Mary Immaculate College, South Circular Road, Limerick, Ireland
}
\email{bernd.kreussler@mic.ul.ie}

\begin{abstract}
  We prove the existence of fine moduli spaces of simple coherent sheaves on
  families of irreducible curves.
  Our proof is based on the existence of a universal upper bound of the
  Castelnuovo-Mumford regularity of such sheaves, which we provide.
\end{abstract}

\maketitle

\section{Introduction}

In \cite{BK-AMS} and \cite{BuH} a geometric
associative (resp.\ classical) $r$-matrix was associated to any flat analytical
family of integral curves of genus one. This way one can construct families of
solutions to the associative (resp.\ classical) Yang-Baxter equation extending
an earlier approach of \cite{Pol1}.

Both constructions are based on the existence of a fine moduli space of simple
vector bundles on an analytical family of integral curves of arithmetic genus
one.

The result of the present article is more general. We prove the
existence of a fine relative moduli space of simple sheaves on any flat
family of integral projective curves of fixed arithmetic genus with a section.
A coherent sheaf $\kF$ on an integral projective curve $E$ over $\CC$ is
called \emph{simple} if the natural linear map
$\CC\rightarrow\Hom_{E}(\kF,\kF)$ is an isomorphism.

The exact assumptions are the following.

\begin{itemize}
\item  Let $p:X \lar S$ be a flat projective
  morphism of complex spaces of relative dimension one and denote by
  $\breve{X}$ the smooth locus of $p$.
\item  Assume there exists a section   $i: S \lar \breve{X}$ of   $p$.
\item  Suppose  that for all points $s \in S$ the fibre $X_s$ is a reduced and
  irreducible projective curve of fixed arithmetic genus $g$.
\end{itemize}

By $\Ans_S$ we denote the category of (not necessarily separated) complex spaces
over $S$. If $f:T\rightarrow S$ is an object of this category, we obtain a
cartesian diagram
\[
\xymatrix
{
X_T \ar[r]^h \ar[d]_{q}& X \ar[d]^p \\
T \ar[r]^f & S\,.
}
\]
For each $t\in T$, we denote by $X_{t}$ the fibre of $q$ over $t$. If
$\kF\in\Coh(X_{T})$, we denote by $\kF_{t}$ the restriction of
$\kF$ to the fibre $X_{t}$.

Given two coprime integers $n\ge0$ and $d$, we define the functor
$\underline{\Mf}^{(n,d)}_{X/S}: \Ans_S \rightarrow \Sets$ as follows:
\[
\underline{\Mf}^{(n,d)}_{X/S}(T \stackrel{f}\to S) =
\left\{
\kF \in \Coh(X_T)\left|
\begin{array}{l}
\kF \text{ is } T\text{-flat, } \kF_{t} \text{ is simple}\\
\rk(\kF_{t})=n, \chi(\kF_{t})=d \quad \forall t \in T
\end{array}
\right.\right\}/\sim
\]
where $\kF_{1} \sim \kF_{2}$  if and only if there exists $\kL \in \Pic(T)$
such that $\kF_{1} \cong \kF_{2} \otimes q^*(\kL)$. By $\rk(\kF_{t})$ we
denote the rank of $\kF_{t}$ at the generic point of the integral curve $X_{t}$.

Our main result, Theorem \ref{T:main}, states that this functor is
representable by a complex space over $S$. For the proof we use a universal
upper bound for the Castelnuovo-Mumford regularity of simple sheaves with
fixed $(n,d)$, provided in Section \ref{S:vanishing}.

Our result may seem to look well known and familiar at the first glance. And
indeed, in the category of schemes, the representability was shown in the
early 1960s by Grothendieck \cite{G} for $(n,d)=(1,d)$.
For general $(n,d)$ Altman and Kleiman \cite{AK1} showed at the
end of the 1970s representability as an algebraic space, but they considered
the sheafification in the \'etale topology of the functor defined above.

Similarly, in the analytic category, localised (or sheafified) functors were
shown to be representable for $S$ a point by Norton \cite{N} in 1979, in the
case $(n,d)=(1,d)$ for general $S$ by Bingener \cite{B} in 1980 and in general
by Kosarew and Okonek \cite{KO} in 1989. For our application, mentioned at the
beginning, we indeed need the representability of the functor
$\underline{\Mf}^{(n,d)}_{X/S}$ as defined above and not only of the
sheafified version of it. Therefore, we need to prove that this functor is
already a sheaf in the analytic topology; this is the main contribution we
make in this paper.

\section{Preliminaries}
\label{S:prelim}

\subsection{Local invariants}\label{SS:local}
A one-dimensional local Noetherian ring is Cohen-Macaulay if its
maximal ideal contains a non-zero divisor.
Let $(A,\mathfrak{m},k)$ be a local Noetherian domain of dimension one.
Then $A$ is Cohen-Macaulay, $\Hom_{A}(k,A)=0$ and $\Ext^{1}_{A}(k,A)\ne0$. The
number
\[\type(A)=\dim_{k}\Ext^{1}_{A}(k,A)\]
is called the \emph{type} of $A$. The ring $A$ is Gorenstein iff $\type(A)=1$,
see \cite[Thm.\ 3.2.10]{BH}.
Let $A\subset\NA\subset Q(A)$ be the normalisation (integral closure) of $A$
in its field of fractions $Q(A)$. Throughout we assume that $A$ is the
localisation of a $k$-algebra of finite type or $k=\CC$ and $A$ is a local
complex analytic algebra. Such rings are excellent and so by \cite[Theorem
6.5]{D}, $\NA$ is a finite $A$-module, which ensures that $\NA/A$ has finite
length.
The \emph{$\delta$-invariant} of $A$ is the
length of the Artinian $A$-module $\NA/A$, i.e.\
\[\delta(A) = \dim_{k}\bigl(\NA/A\bigr)\,.\]
\begin{lemma}\label{L:type-delta}
  If $A$ is not regular, then $\type(A)\le\delta(A)$.
\end{lemma}
\begin{proof}
  Applying $\Hom_{A}(\ARG,A)$ to the exact sequence
  $0\rightarrow \mathfrak{m}\rightarrow A \rightarrow k \rightarrow 0$
  yields the exact sequence of $A$-modules
  \[0\lar A\lar \Hom_{A}(\mathfrak{m},A) \lar \Ext^{1}_{A}(k,A)\lar 0\,.\]
  We would like to compare it with the exact sequence
  $0\rightarrow A\rightarrow\NA\rightarrow\NA/A\rightarrow 0$.
  To this end, we are going to establish an
  injective $A$-linear map $\Hom_{A}(\mathfrak{m},A)\rightarrow\NA$.

  Observe first that the maps
  $\alpha:\Lambda=\{\lambda\in Q(A)\mid \lambda\mathfrak{m}\subset\mathfrak{m}\}
  \lar \Hom_{A}(\mathfrak{m},\mathfrak{m})$,
  given by $\alpha(\lambda)(x)=\lambda x$, and
  $\beta:\Hom_{A}(\mathfrak{m},\mathfrak{m}) \lar \Hom_{A}(\mathfrak{m},A)$,
  induced by
  $\mathfrak{m}\subset A$, are both isomorphisms of $A$-modules.
  Indeed, both are injective and, because $Q(A)$ is a
  field, the surjectivity of $\alpha$ follows from
  $\mathfrak{m}\subset\mathfrak{m}\otimes_{A}Q(A)\cong Q(A)$.
  The surjectivity of
  $\beta$ uses that $A$ is not regular as follows. Every $A$-linear map
  $f:\mathfrak{m}\rightarrow A$ is injective (as $A$ is a domain) and if it
  does not factor via $\mathfrak{m}\subset A$, it is surjective. But if $f$ is
  an isomorphism, $\mathfrak{m}$ is a principal ideal, which implies that $A$
  is regular.

  To conclude, note that \cite[Prop.\ 2.4]{AM} implies that $\Lambda\subset
  Q(A)$ is actually contained in $\NA$. Together with the isomorphisms
  $\alpha$ and $\beta$ this gives the desired inclusion of $A$-modules
  $\Hom_{A}(\mathfrak{m},A)\subset\NA$. Using the above exact sequence, this
  inclusion induces an injective $A$-linear map
  $\Ext^{1}_{A}(k,A)\rightarrow \NA/A$, hence
  \(\type(A)\le \delta(A)\,.\)
\end{proof}
\begin{remark}
  If $A$ is regular, $\type(A)=1$ and $\delta(A)=0$.
\end{remark}

Recall that the conductor ideal of $(A,\mathfrak{m},k)$ is
\[\CI=\Ann_{A}\bigl(\NA/A\bigr) = \{\lambda\in A\mid
\lambda\cdot\NA\subset A\}\,.\]
This is an ideal in $A$ as well as in $\NA$, hence $\CI\cdot\NA=\CI$ and so
$\NA/\CI \cong A/\CI\otimes_{A}\NA$. Moreover, the map
$\CI\lar\Hom_{A}(\NA,A)$ which sends $\lambda\in\CI$ to multiplication by
$\lambda$, is an isomorphism of $A$-modules: it is injective as $A$ is a
domain and to see surjectivity we use that $\NA\otimes_{A}Q(A)\cong Q(A)$.
If $A$ is not regular, $\CI\ne A$ and so $\CI\subset\mathfrak{m}$. The
\emph{multiplicity of the conductor} is the length of the Artinian $A$-module
$\NA/\CI$, denoted
\[c(A) = \dim_{k}\NA/\CI\,.\]
Because normalisation and completion are commuting processes for reduced
excellent rings \cite[Theorem 6.5]{D}, the invariants $\type(A)$,
$\delta(A)$ and $c(A)$ are the same for the localisation $A$ of a
$\mathbb{C}$-algebra of finite type and the analytic local rings given by
the same ideal. Therefore, we can use the work of Greuel \cite{GMG} in our
situation to obtain the following inequality.
\begin{lemma}\label{L:GMG}{\rm{\cite[2.4 (d)]{GMG}}}
  If the one-dimensional local Noetherian domain $A$ is the localisation of a
  $\mathbb{C}$-algebra of finite type, then  $c(A)\le 2\delta(A)$.
\end{lemma}

\subsection{Serre duality}\label{SS:serre-duality}
Let $E$ be an integral projective curve. From the previous subsection we know
that $E$ is Cohen-Macaulay.
Therefore, on $E$ there exists a dualising sheaf, which we denote by
$\omega_{E}$, and there exist functorial isomorphisms of vector
spaces
\[\Ext^{1}(\kF,\omega_{E})\cong H^{0}(\kF)^{\ast}
\qquad\text{ and }\qquad
\Hom(\kF,\omega_{E})\cong H^{1}(\kF)^{\ast}\]
for all coherent sheaves $\kF$ on $E$, see \cite[III, \S7]{H}.

\subsection{Relative duality}\label{SS:duality}
We also need relative duality for the normalisation map
$\nu:\widetilde{E}\lar E$ of an integral projective curve $E$.
Because $\nu$ is a finite morphism, what we need is contained in \cite[III,
Ex.\ 6.10]{H} and worked out in detail in \cite{BC}.
\begin{itemize}
\item[(a)] For each quasi-coherent $\kO_{E}$-module $\kG$ there exists a
  unique (up to isomorphism) $\kO_{\widetilde{E}}$-module $\nu^{!}\kG$ such that
  \(\nu_{\ast}\nu^{!}\kG \cong
  \ShHom_{E}(\nu_{\ast}\kO_{\widetilde{E}},\kG)\).
  If $\kG$ is locally free, the projection formula implies
  $\nu^{!}\kG \cong \nu^{\ast}\kG\otimes\nu^{!}\kO_{E}$.
\item[(b)] For any coherent $\kO_{\widetilde{E}}$-module $\kF$ and any
  quasi-coherent $\kO_{E}$-module $\kG$ there exists a natural isomorphism
  \(\nu_{\ast}\ShHom_{\widetilde{E}}(\kF,\nu^{!}\kG)
  \stackrel{\sim}{\lar}
  \ShHom_{E}(\nu_{\ast}\kF,\kG)\). Taking global sections, we obtain
  \[\Hom_{\widetilde{E}}(\kF,\nu^{!}\kG)
  \cong
  \Hom_{E}(\nu_{\ast}\kF,\kG)\,.\]
\end{itemize}

\subsection{Global invariants}\label{SS:global}
Let $E$ be an integral projective curve over $\CC$ and
$\nu:\widetilde{E}\rightarrow E$ its normalisation. By $p_{a}(E)=h^{1}(\kO_{E})$
and $p_{g}(E)=p_{a}(\widetilde{E})$ we denote the arithmetic and geometric
genus of $E$, respectively.
To define global versions of the local invariants introduced above, we
look at the exact sequence
\[
  0\lar\kO_{E} \lar \nu_{\ast}\kO_{\widetilde{E}} \lar \kT \lar 0
\]
in which $\kT$ is a torsion sheaf with support in the singular locus of $E$.
If $x\in \sing(E)$ and $A=\kO_{E,x}$, this sequence localises at
$x$ to $0\lar A\lar\NA\lar\NA/A\lar 0$. Hence, $\dim_{\mathbb{C}}\kT_{x} =
\delta(\kO_{E,x})$ and we obtain
\[\delta(E):=\sum_{x\in \sing(E)} \delta(\kO_{E,x}) = \chi(\kT)
= \chi\bigl(\nu_{\ast}\kO_{\widetilde{E}}\bigr) - \chi\bigl(\kO_{E}\bigr)
= p_{a}(E)-p_{g}(E) \le p_{a}(E)\,.\]
Because $\delta(E)\ge 0$, we have $p_{g}(E)\le p_{a}(E)$.
The conductor ideal sheaf is defined as
\[\kC = \kA nn_{\kO_{E}}\bigl(\nu_{\ast}\kO_{\widetilde{E}}/\kO_{E}\bigr)\,.\]
The ideal sheaf
$\kC\subset\kO_{E}$ defines a subscheme $Z\subset E$ which is supported in
the singular locus $\sing(E)$ of $E$. Let $
\widetilde{Z}=Z\times_{E}\widetilde{E}$, which is a subscheme of
$\widetilde{E}$ whose ideal sheaf is denoted by $\widetilde{\kC}$.
From our local considerations, we know
$\NA/\CI \cong A/\CI\otimes_{A}\NA$ and $\CI\cdot\NA=\CI$. This implies that
$\widetilde{\kC}$ is the image of the canonical map $\nu^{\ast}\kC \lar
\kO_{\widetilde{E}}$. This gives us a canonical map $\kC \lar
\nu_{\ast}\nu^{\ast}\kC \lar \nu_{\ast}\widetilde{\kC}$, which locally
coincides with the composition $\CI \lar \CI\otimes_{A}\NA \lar \CI$, which is
equal to the identity. Hence, we have a canonical isomorphism
$\kC\stackrel{\sim}{\lar}\nu_{\ast}\widetilde{\kC}$.
Because
$\kC \cong \ShHom_{E}(\nu_{\ast}\kO_{\widetilde{E}}, \kO_{E})$, which we have
seen locally, the definition of $\nu^{!}$ implies
\[\nu^{!}\kO_{E} \cong \widetilde{\kC}\,.\]

\begin{lemma}\label{C:chi-conductor}
  $\chi\bigl(\nu^{!}\kO_{E}\bigr) \ge 1+p_{g}(E)-2p_{a}(E) \ge 1-2p_{a}(E)$
\end{lemma}
\begin{proof}
  Using the exact sequence
  $0\rightarrow\widetilde{\kC}\rightarrow \kO_{\widetilde{E}}
  \rightarrow \kO_{\widetilde{Z}} \rightarrow 0$, we obtain
  \[\chi\bigl(\nu^{!}\kO_{E}\bigr) = \chi\bigl(\widetilde{\kC}\bigr)
  = \chi\bigl(\kO_{\widetilde{E}}\bigr) - \chi\bigl(\kO_{\widetilde{Z}}\bigr)
  = 1-p_{g}(E) - \chi\bigl(\kO_{\widetilde{Z}}\bigr)\,.\]
  Because $\chi\bigl(\kO_{\widetilde{Z}}\bigr)
  = \sum_{x\in \sing(E)} c\bigl(\kO_{E,x}\bigr)$, Lemma \ref{L:GMG}
  implies
  \[\chi\bigl(\kO_{\widetilde{Z}}\bigr)
  \le 2\sum_{x\in \sing(E)} \delta(\kO_{E,x})
  = 2\delta(E) = 2p_{a}(E)-2p_{g}(E)\,,\]
  hence $\chi\bigl(\nu^{!}\kO_{E}\bigr) \ge 1+p_{g}(E)-2p_{a}(E) \ge
  1-2p_{a}(E)$.
\end{proof}

\subsection{Families of simple sheaves}\label{SS:simple}

The families of sheaves considered in the definition of the functor
$\underline{\Mf}^{(n,d)}_{X/S}$ are simple on each fibre. We also need
that for such families the canonical morphism
\(\kO_{S}\rightarrow p_{\ast}\ShHom_{X}(\kF,\kF)\) is an isomorphism. This can
be found in \cite[Lemma 4.6.3]{HL} and in \cite[Cor.\ 5.3]{AK1} for the
case of schemes, which inspired the proof we present for the analytic case.

\begin{lemma}\label{L:simple}
  If $p:X \lar S$ be a proper and flat morphism of complex spaces and
  $\kF$ a coherent sheaf on $X$, flat over $S$, such that $\kF_{s}$ is simple
  for all $s \in S$, then \[\kO_{S}\stackrel{\sim}\lar
  p_{\ast}\ShHom_{X}(\kF,\kF)\,.\]
\end{lemma}
\begin{proof}
  The canonical morphism $\kO_{S}\rightarrow p_{\ast}\ShHom_{X}(\kF,\kF)$ is
  injective because $\kF$ is $S$-flat. Restricting the exact sequence
  \(0\rightarrow \kO_{S} \rightarrow p_{\ast}\ShHom_{X}(\kF,\kF) \rightarrow
  \kT \rightarrow 0\) to a point $s\in S$, we obtain the exact sequence
  \[\CC \stackrel{\alpha}\lar p_{\ast}\ShHom_{X}(\kF,\kF)(s)
  \lar \kT(s) \lar 0\,.\]
  The base-change morphism
  $\varphi_{s}:p_{\ast}\ShHom_{X}(\kF,\kF)(s) \rightarrow
  \Hom_{X_{s}}(\kF_{s},\kF_{s})$ composed with $\alpha$ yields the canonical map
  $\CC\rightarrow\Hom_{X_{s}}(\kF_{s},\kF_{s})$, which was supposed to be an
  isomorphism. By \cite[Satz 2 (ii)]{BPS} the surjectivity of $\varphi_{s}$
  already implies that it is an isomorphism. Hence, $\alpha$ is an isomorphism
  as well and so $\kT(s)=0$ for all $s\in S$. As all local rings $\kO_{S,s}$
  are Noetherian, this is sufficient to conclude that $\kT=0$, which implies
  the claim.
\end{proof}
\section{A vanishing theorem for simple sheaves}
\label{S:vanishing}

Serre's Vanishing Theorem says that twists of a coherent sheaf with a
sufficiently high tensor power of an ample line bundle have vanishing
higher cohomology. How large this tensor power needs to be, in general depends
on the underlying projective variety and the coherent sheaf. We will prove
here a version for simple sheaves on complex projective curves in which the
vanishing is guaranteed above a threshold which only depends on the arithmetic
genus of the curve as well as the rank and the Euler characteristic of the
coherent sheaf.

Throughout this section, we fix integers $g\ge0, n\ge0$ and $d$. All curves
will be curves over the field $\CC$.
The main result of this section is the following theorem.

\begin{theorem}\label{T:vanishing}
  There exists an integer $m_{0}=m_{0}(g,n,d)$, depending only on $g, n, d$,
  such that for all integral projective curves $E$ of arithmetic genus $g$,
  for all simple coherent sheaves $\kF$ on $E$ with $\rk(\kF)=n$ and
  $\chi(\kF)=d$, for all line bundles $\kL$ of degree $1$ on $E$ and for all
  $m\ge m_{0}$ we have
  \[H^{1}(\kF\otimes\kL^{\otimes m}) = 0.\]
\end{theorem}

As a tool in our proofs we use a stability structure on the abelian category
of coherent sheaves on an integral projective curve $E$. It is convenient to
use the framework of \cite{Rudakov}, where stability is defined with the aid
of a linearly independent system of additive functions. We exclusively work
with the pair of additive functions $(\rk,\chi)$. We define
$\mu(\kF)=\chi(\kF)/\rk(\kF)$ when $\rk(\kF)>0$ and let $\mu(\kF)=\infty$ if
$\rk(\kF)=0$. A non-zero coherent sheaf $\kF$ is called semi-stable if
$\mu(\kG)\le\mu(\kF)$ for all non-zero subsheaves $\kG$.
In particular, all torsion sheaves are semi-stable and semi-stable sheaves of
positive rank are automatically torsion free. Moreover, it is not hard to see
that all torsion free sheaves of rank one are semi-stable.
The key facts we need are
\begin{itemize}
\item[(i)] If $\kF, \kG$ are semi-stable and $\mu(\kF)>\mu(\kG)$ then
  $\Hom(\kF,\kG)=0$.
\item[(ii)] For every non-zero coherent sheaf $\kF$ there exists a filtration,
  known as the Harder-Narasimhan filtration (HNF),
  $0=\kF_{0}\subset\kF_{1}\subset\ldots\subset\kF_{k-1}\subset\kF_{k}=\kF$,
  such that the factors $\kA_{i}=\kF_{i}/\kF_{i-1}$ are semi-stable and
  \[\mu_{\text{max}}(\kF)=\mu(\kA_{1})>\mu(\kA_{2})>\ldots
  >\mu(\kA_{k-1})>\mu(\kA_{k})=\mu_{\text{min}}(\kF)\,.\]
  This filtration is unique and so $\mu_{\text{max}}$ and $\mu_{\text{min}}$
  are well-defined. We easily obtain
  \(\mu_{\text{max}}(\kF) \ge \mu(\kF) \ge \mu_{\text{min}}(\kF)\).
\end{itemize}
Proofs are pretty well-known by now and can be found in \cite{Rudakov}.
Note that the abelian category of coherent sheaves on $E$ is
Noetherian (i.e.\ ascending chains of subobjects eventually stabilise) and,
with respect to the stability defined above, it is weakly Artinian in the sense
of \cite{Rudakov}, i.e.\ there is no infinite descending chain of coherent
subsheaves with strictly increasing slopes.

\begin{remark}
  The vanishing in Theorem \ref{T:vanishing} is easily obtained for
  semi-stable sheaves $\kF$. If $\kF$ is torsion, the vanishing is trivial. If
  $\kF$ is torsion free, we set $m_{0}=g-1-\frac{d}{n}$. Then we have
  $\mu(\kF)=\frac{d}{n}> 2g-2-m+1-g =
  \mu\left(\omega_{E}\otimes\kL^{\otimes(-m)}\right)$, whenever $m>m_{0}$.
  Because $\omega_{E}$ is torsion free of rank one, for $\kF$ semi-stable
  we obtain $\Hom\left(\kF, \omega_{E}\otimes\kL^{\otimes(-m)}\right) = 0$.
  Using Serre-duality this implies
  \[H^{1}(\kF\otimes\kL^{\otimes m})
  = \Hom(\kF\otimes\kL^{\otimes m},\omega_{E})^{\ast}
  = \Hom(\kF, \omega_{E}\otimes\kL^{\otimes(-m)})^{\ast} = 0.\]
  If $\kF$ is not semi-stable, we need to guarantee such a vanishing for all
  semi-stable quotients that can appear in the Harder-Narasimhan filtration of
  $\kF$, but their slopes are no longer equal to $\frac{d}{n}$. Our main task
  will be to find a lower bound for these slopes.
\end{remark}

\begin{lemma}\label{L:slope}
  Let $E$ be an integral projective curve, $\kF\in\Coh(E)$ a torsion free
  sheaf and $\kL\in\Pic(E)$, then \[\mu(\kF\otimes\kL)=\mu(\kF)+\deg(\kL).\]
\end{lemma}

\begin{proof}
  Because $\rk(\kF)>0$, the statement is equivalent to the equation
  \begin{equation}
    \label{eq:chi}
    \chi(\kF\otimes\kL) = \chi(\kF)+\deg(\kL)\rk(\kF).
  \end{equation}
  This equation holds true for arbitrary coherent sheaves on $E$. It is
  obvious if $\kF$ is a torsion sheaf. Because $\chi$ and $\rk$ are additive
  functions, both sides of equation \eqref{eq:chi} are additive. Looking at
  the exact sequence
  $0 \rightarrow \tors(\kF) \rightarrow \kF \rightarrow \kG \rightarrow 0$,
  where $\kG$ is a torsion free sheaf and $\tors(\kF)$ is the
  torsion subsheaf of $\kF$, we see that it is sufficient to prove equality
  \eqref{eq:chi} for torsion free sheaves.

  If $\kF$ is torsion free and $\nu:\widetilde{E}\rightarrow E$ is the
  normalisation, we let $\widetilde{\kF}=\nu^{\ast}(\kF)/\text{torsion}$ and
  obtain an exact sequence $0\rightarrow \kF \rightarrow
  \nu_{\ast}\widetilde{\kF} \rightarrow \kT \rightarrow 0$, where $\kT$ is a
  torsion sheaf on $E$. Hence it is enough to show equation \eqref{eq:chi} for
  sheaves $\kF=\nu_{\ast}\widetilde{\kF}$, where $\widetilde{\kF}$ is a vector
  bundle on the smooth curve $\widetilde{E}$. The projection formula gives an
  isomorphism $\nu_{\ast}\widetilde{\kF}\otimes\kL \cong
  \nu_{\ast}\bigl(\widetilde{\kF}\otimes\nu^{\ast}\kL\bigr)$. The usual
  Riemann-Roch for vector bundles on $\widetilde{E}$ implies
  \begin{align*}
    \chi\bigl(\nu_{\ast}\widetilde{\kF}\otimes\kL\bigr)
    &=
    \chi\bigl(\nu_{\ast}\bigl(\widetilde{\kF}\otimes\nu^{\ast}\kL\bigr)\bigr)
    = \chi\bigl(\widetilde{\kF}\otimes\nu^{\ast}\kL\bigr)
    = \chi\bigl(\widetilde{\kF}\bigr) +
    \deg\bigl(\nu^{\ast}\kL\bigr)\rk\bigl(\widetilde{\kF}\bigr)\\
    &= \chi\bigl(\nu_{\ast}\widetilde{\kF}\bigr) +
    \deg\bigl(\kL\bigr)\rk\bigl(\nu_{\ast}\widetilde{\kF}\bigr)
  \end{align*}
  from which the claim follows.
\end{proof}

\begin{remark}
  Because tensor product with a line bundle $\kL$ is an exact functor,
  the HN factors of $\kF\otimes\kL$ are $\kA_{i}\otimes\kL$, when
  $\kA_{i}$ are the HN factors of $\kF$. As a consequence, the equation in
  Lemma \ref{L:slope} holds for $\mu_{\text{max}}$ and $\mu_{\text{min}}$
  as well.
\end{remark}

\begin{lemma}\label{L:small-slope}
  Let $E$ be an integral projective curve, $\kA\in\Coh(E)$ torsion free,
  $q\in\mathbb{Z}$ and $\kL\in\Pic(E)$ such that $\mu(\kA)<-q\deg(\kL)$, then
  $\Hom(\kA,\omega_{E}\otimes\kL^{-q}) \ne 0.$
\end{lemma}

\begin{proof}
  By Lemma \ref{L:slope} we have $\mu(\kA\otimes\kL^{q}) =
  \mu(\kA)+q\deg(\kL)<0$, hence $\chi(\kA\otimes\kL^{q})<0$ and this implies
  $0\ne H^{1}(\kA\otimes\kL^{q})
  = \Hom(\kA\otimes\kL^{q},\omega_{E})^{\ast}
  =\Hom(\kA,\omega_{E}\otimes\kL^{-q})^{\ast}$, using Serre-duality.
\end{proof}

\begin{lemma}\label{L:omega}
  Let $g\ge 0$ be an integer. Then, for all integral projective curves $E$ of
  arithmetic genus $g$, for all line bundles $\kL$ of degree one on $E$ and
  for all $q\ge g^{2}+4g-2$
  \[\Hom(\omega_{E},\kL^{q})\ne 0\,.\]
\end{lemma}

\begin{proof}
  Let $\nu:\widetilde{E}\rightarrow E$ be the normalisation.
  As $\nu^{!}\omega_{E}\cong\omega_{\widetilde{E}}$ (see \cite[III Ex.\
  7.2]{H}) and $(\nu_{\ast},\nu^{!})$ is an adjoint pair, there is a natural
  morphism of $\kO_{E}$-modules
  $\nu_{\ast}\omega_{\widetilde{E}} \rightarrow \omega_{E}$. Because
  $\omega_{\widetilde{E}}$ is torsion free and the natural morphism is an
  isomorphism on the regular locus of $E$, we obtain an exact sequence
  \begin{equation}
    \label{eq:T}
    0\lar\nu_{\ast}\omega_{\widetilde{E}} \lar \omega_{E} \lar \kT \lar 0
  \end{equation}
  in which $\kT$ is a torsion sheaf with support in the singular locus of $E$.
  Applying the functor $\Hom(\ARG,\kL^{q})$ produces the exact sequence
  \begin{equation}
    \label{eq:sequ}
    0\lar\Hom(\omega_{E},\kL^{q})
    \lar\Hom(\nu_{\ast}\omega_{\widetilde{E}},\kL^{q})
    \lar\Ext^{1}(\kT,\kL^{q}).
  \end{equation}
  We first give an upper bound, not depending on $q$, for the dimension
  of the third term in this sequence. Then we will show that the dimension of
  the second term grows unboundedly with $q$.

  For each singular point $x\in E$ we
  denote by $\ell_{x}$ the dimension of the stalk $\kT_{x}$, so that
  $\chi(\kT) = \sum_{x\in \sing(E)} \ell_{x}$.
  Using a Jordan-H\"{o}lder filtration for the $\kO_{E,x}$-module $\kT_{x}$,
  we obtain
  $\dim \ShExt^{1}(\kT,\kL^{q})_{x} = \dim \Ext^{1}_{\kO_{E,x}}(\kT_{x},\kO_{E,x})
  \le \ell_{x}\cdot\type(\kO_{E,x})$.
  Because $\Ext^{1}(\kT,\kL^{q}) = H^{0}(\ShExt^{1}(\kT,\kL^{q}))$,
  Lemma \ref{L:type-delta} implies now
  \[\dim\Ext^{1}(\kT,\kL^{q})
  \le \sum_{x\in \sing(E)} \ell_{x}\delta\left(\kO_{E,x}\right)
  \le \delta(E)\sum_{x\in \sing(E)} \ell_{x}
  = \delta(E)\chi(\kT).\]
  From \eqref{eq:T} we get  $\chi(\kT)
  = \chi(\omega_{E}) - \chi(\nu_{\ast}\omega_{\widetilde{E}})
  = p_{a}(E) - p_{g}(E) \le p_{a}(E)=g$.
  We have seen in Subsection \ref{SS:global} that $\delta(E)\le p_{a}(E)=g$ and
  therefore we can conclude
  \begin{equation}
    \label{eq:ext}
    \dim\Ext^{1}(\kT,\kL^{q}) \le  g^{2}\,.
  \end{equation}

  To study the second term in the sequence \eqref{eq:sequ} we
  first use relative duality to obtain
  \begin{align*}
    \Hom_{E}(\nu_{\ast}\omega_{\widetilde{E}},\kL^{q})
    &= \Hom_{\widetilde{E}}(\omega_{\widetilde{E}},
    \nu^{!}\kL^{q})
    = \Hom_{\widetilde{E}}(\omega_{\widetilde{E}},
    \nu^{\ast}\kL^{q}\otimes\nu^{!}\kO)\\
    &=
    H^{0}(\nu^{\ast}\kL^{q}\otimes\omega_{\widetilde{E}}^{\vee}\otimes\nu^{!}\kO)
    \,.
  \end{align*}
  Because $\nu^{\ast}\kL^{q}\otimes\omega_{\widetilde{E}}^{\vee}$ is an
  invertible sheaf and $\nu^{!}\kO$ has rank one, with the aid of Lemma
  \ref{L:slope} and Lemma \ref{C:chi-conductor}, and using
  $g=p_{a}(E)\ge p_{g}(E)$, we obtain
  \begin{align*}
    h^{0}\bigl(\nu^{\ast}\kL^{q}\otimes\omega_{\widetilde{E}}^{\vee}
    \otimes\nu^{!}\kO\bigr) &\ge
    \chi\bigl(\nu^{\ast}\kL^{q}\otimes\omega_{\widetilde{E}}^{\vee}
    \otimes\nu^{!}\kO\bigr) =
    \chi(\nu^{!}\kO) +
    \deg\bigl(\nu^{\ast}\kL^{q}\otimes\omega_{\widetilde{E}}^{\vee}\bigr)\\
    &=
    \chi(\nu^{!}\kO) + q-2p_{g}(E)+2
    \ge 1-2p_{a}(E)+ q-2p_{g}(E)+2\\
    &\ge q+3-4g \ge g^{2}+1
  \end{align*}
  whenever $q\ge g^{2}+4g-2$. The sequence \eqref{eq:sequ} implies now that
  $\Hom(\omega_{E},\kL^{q})\ne 0$ for all $q\ge g^{2}+4g-2$.
\end{proof}

\begin{proposition}\label{P:difference}
  Let $E$ be an integral projective curve of arithmetic genus $g$ and let
  $\kF\in\Coh(E)$ be a simple torsion free sheaf, then
  \[\mu(\kF)-\mu_{\text{min}}(\kF)\le g^{2}+4g\,.\]
\end{proposition}

\begin{proof}
  First observe that Lemma \ref{L:slope} implies that
  $\mu(\kF)-\mu_{\text{min}}(\kF)$ remains unchanged when $\kF$ is twisted by
  a line bundle. After twisting with a line bundle we can always achieve that
  $0<\mu_{\text{max}}(\kF)\le 1$. Hence, we can make this assumption without
  loss of generality. Let $0=\kF_{0}\subset
  \kF_{1}\subset \ldots \subset \kF_{k-1}\subset \kF_{k}=\kF$ be the HNF of
  $\kF$ with semi-stable factors $\kA_{i}=\kF_{i}/\kF_{i-1}$. We have
  $\mu_{\text{max}}(\kF) = \mu(\kA_{1})$ and $\mu_{\text{min}}(\kF) =
  \mu(\kA_{k})$.

  For a proof by contradiction we assume $\mu(\kA_{k})<1-4g-g^{2}$.
  Let $\kL$ be a line bundle of degree one on $E$.
  By Lemma \ref{L:small-slope} with $q=g^{2}+4g-1$ there
  exists a non-vanishing morphism of sheaves $\alpha:\kA_{k}\rightarrow
  \omega_{E}\otimes\kL^{1-4g-g^{2}}$. From Lemma \ref{L:omega} with
  $q=g^{2}+4g-1>g^{2}+4g-2$ we obtain a non-zero morphism
  $\beta:\omega_{E}\otimes\kL^{1-4g-g^{2}} \rightarrow \kO$. Finally, because
  we assumed $\mu(\kA_{1}) = \mu_{\text{max}}(\kF)>0$, we have
  $h^{0}(\kA_{1})\ge \chi(\kA_{1}) >0$ and so there exists a non-zero morphism
  $\gamma:\kO\rightarrow \kA_{1}$.
  Together with the epimorphism $\kF\twoheadrightarrow\kA_{k}$ and the inclusion
  $\kA_{1}\subset\kF$, we obtain a chain of non-zero morphisms
  \begin{equation}
    \label{eq:chain}
    \kF\twoheadrightarrow\kA_{k}
    \stackrel{\alpha}{\lar}\omega_{E}\otimes\kL^{1-4g-g^{2}}
    \stackrel{\beta}{\lar}\kO
    \stackrel{\gamma}{\lar}\kA_{1}\subset\kF.
  \end{equation}
  Because any morphism from a torsion free rank one sheaf to a torsion free
  sheaf is automatically a monomorphism, $\beta$ and $\gamma$ are monomorphism
  and so the composition \eqref{eq:chain} is a non-zero endomorphism of
  $\kF$. As $\kF$ was assumed to be simple, this composition must be an
  isomorphism and so also each morphism in the chain \eqref{eq:chain}.
  In particular, $\kF\cong\kO$ and so
  $\mu_{\text{min}}(\kF) = \mu(\kF) = \chi(\kO) = 1-g \ge 1-4g-g^{2}$. This
  contradicts our assumption $\mu(\kA_{k})<1-4g-g^{2}$, thus proves that we
  have $\mu_{\text{min}}(\kF)=\mu(\kA_{k})\ge 1-4g-g^{2}$. Using
  $\mu(\kF)\le\mu_{\text{max}}(\kF)\le 1$, we obtain
  $\mu(\kF)-\mu_{\text{min}}(\kF)\le g^{2}+4g$, as required.
\end{proof}

\begin{proof}[Proof of Theorem \ref{T:vanishing}]
  Let $E$, $\kF$ and $\kL$ be as in the formulation of the theorem.
  If the simple sheaf $\kF$ is not torsion free, the torsion subsheaf of $\kF$
  is non-zero and so contains a rank one sky-scraper $\kk_{s}$ supported at some
  point $s\in E$. On the other hand, by restricting to $s$ we always get an
  epimorphism $\kF\twoheadrightarrow \kk_{s}$. The composition
  $\kF\twoheadrightarrow \kk_{s} \subset \kF$ is a non-trivial endomorphism,
  hence an isomorphism. In particular, $\kF \cong \kk_{s}$ is a torsion
  sheaf. In this case, the claimed vanishing is trivial for all
  $m\in\mathbb{Z}$, because the support of $\kF\otimes\kL^{\otimes m}$ is of dimension
  zero.

  For the rest of the proof we assume that $\kF$ is a simple torsion free sheaf.
  Let $0=\kF_{0}\subset \kF_{1}\subset \ldots \subset \kF_{k-1}\subset
  \kF_{k}=\kF$ be the HNF of $\kF$ with semi-stable factors
  $\kA_{i}=\kF_{i}/\kF_{i-1}$. Recall that $\mu(\kF)=\frac{d}{n}$.
  Let $m_{0}=\lfloor g^{2}+5g-\frac{d}{n}\rfloor$ and $m\ge m_{0}$.
  Because $\mu(\omega_{E}\otimes\kL^{\otimes (-m)}) = g-1-m \le g-1-m_{0} <
  g-1-(g^{2}+5g-1-\mu(\kF)) = \mu(\kF)-(g^{2}+4g)\le \mu_{\text{min}}(\kF)$ by
  Proposition \ref{P:difference}, the semi-stability of the factors $\kA_{i}$
  implies
  $\Hom(\kA_{i}, \omega_{E}\otimes\kL^{\otimes (-m)}) = 0$. From the exact sequences
  \(0\rightarrow\kF_{i-1}\rightarrow\kF_{i}\rightarrow\kA_{i}\rightarrow 0\)
  we obtain then $\Hom(\kF, \omega_{E}\otimes\kL^{\otimes{(-m)}}) = 0$. Serre-duality
  implies now
  \(H^{1}(\kF\otimes\kL^{\otimes m})
  = \Hom(\kF\otimes\kL^{\otimes m},\omega_{E})^{\ast}
  = \Hom(\kF, \omega_{E}\otimes\kL^{\otimes{(-m)}})^{\ast} = 0\), as required.
\end{proof}

\begin{remark}
  A moduli problem for coherent sheaves in the category of schemes will be
  bounded if the Castelnuovo-Mumford regularity of these sheaves is bounded.
  The corresponding moduli space is then of finite type or even
  quasi-projective. The significance of Theorem \ref{T:vanishing} in the
  present situation is highlighted in Remark \ref{R:bounded}.
\end{remark}

\section{Proof of the main theorem}
\label{S:classical}

The family $p:X\rightarrow S$ is assumed to satisfy the conditions formulated
in the Introduction. In addition to the notation introduced there we denote by
$\Sigma = i(S)$ the image of the section $i$, which is a divisor on $X$. For
any object $T\rightarrow S$ in $\Ans_{S}$, we continue to denote the
projection by $q:X_{T}\rightarrow T$ and denote by $\Sigma_{T}\subset X_{T}$
the induced section. For any coherent sheaf $\kF$ on $X_{T}$ we write
$\kF(m):=\kF\otimes\kO(m\Sigma_{T})$.
\begin{theorem}\label{T:main}
  If $n$ and $d$ are coprime, then the functor $\underline{\Mf}^{(n,d)}_{X/S}$
  is representable.
\end{theorem}

In order to prove the existence of a fine moduli space $\Mf^{(n,d)}_{X/S}$ we
apply a result of Kosarew and Okonek \cite{KO}. If we can show that the
functor $\underline{\Mf}^{(n,d)}_{X/S}$ is a sheaf in the analytic topology,
it follows from their Theorem 6.4 and Remark 6.7 that this functor is
representable.

The separatedness of the functor $\underline{\Mf}^{(n,d)}_{X/S}$, which is
part of the sheaf property, means that for any morphism $T\rightarrow S$,
two families $\kF, \kG \in \underline{\Mf}^{(n,d)}_{X/S}(T)$ coincide if
they do so locally. More precisely, this means that if for an open cover
$\{T_{i}\}$ of $T$ we have $\kF|_{X_{T_{i}}} \sim \kG|_{X_{T_{i}}}$,
then $\kF \sim \kG$.
The following lemma
implies that the functor $\underline{\Mf}^{(n,d)}_{X/S}$ is separated.
The equivalence relation $\sim$ corresponds precisely to the
construction of a separated functor, see \cite[Remark (6.7)]{KO}, or
\cite[Section 2.3]{FGA}.

\begin{lemma}\label{L:wellknown}
  Let $p:X \lar S$ be a proper and flat morphism of complex spaces,
  $\kF$ and $\kG$ two $S$-flat coherent sheaves on $X$ such that for all
  points $s \in S$
  \begin{enumerate}
  \item[(i)] $\kF_{s}$ is simple;
  \item[(ii)] there exists an open neighbourhood $U$ of $s\in S$ such that
    $\kF_{U} \cong \kG_{U}$;
  \end{enumerate}
  then there exists a line bundle
  $\kL \in \Pic(S)$ such that $\kG \cong \kF \otimes p^*\kL$.
\end{lemma}

\begin{proof}
  Assumption (ii), in which $\kF_{U}, \kG_{U}$ denote restrictions to the open
  subset $X_{U}\subset X$, implies that for any $s \in S$ there exists an
  isomorphism $\varphi: \kF_{s} \rightarrow \kG_{s}$ which
  generates the one-dimensional vector space $\Hom_{X_s}(\kF_{s},\kG_{s})$. We
  obtain a commutative diagram
  \[
  \xymatrix
  {
    \kF_{s} \ar[rr] \ar[d]_{\varphi}
    & & \kF_{s} \otimes
    \bigl(\Hom_{X_s}(\kF_{s},\kG_{s}) \otimes_{\CC} \kO_{X_{s}}\bigr)
    \ar[d]^{\id \otimes \ev}
    \\
    \kG_{s}  & &  \kF_{s} \otimes \ShHom_{X_s}(\kF_{s},
    \kG_{s}) \ar[ll]_(.64){\ev}
  }
  \]
  in which the upper horizontal arrow is the isomorphism induced by
  $\varphi$. This implies that the canonical morphism
  \[
  \kF_{s} \otimes \bigl(\Hom_{X_s}(\kF_{s},
  \kG_{s}) \otimes_{\CC} \kO_{X_s}\bigr) \lar
  \kG_{s}
  \]
  is an isomorphism for all $s \in S$.

  Because
  $p_*\ShHom_X(\kF, \kG)|_{U} \cong p_{U\ast}\ShHom_{X_{U}}(\kF_{U},
  \kG_{U})$,
  where $p_{U}:X_{U}\rightarrow U$ is the restriction of $p$ to $U\subset S$,
  assumptions (i) and (ii) and Lemma \ref{L:simple} imply that
  $p_*\ShHom_X(\kF, \kG)|_{U} \cong \kO_{U}$. This means that
  $\kL : = p_*\ShHom_X(\kF, \kG)$ is a line bundle on $S$.
  Thus, the following composition of canonical morphisms of sheaves on $X$
  \[
  \kF \otimes p^* p_* \ShHom_X(\kF, \kG) \lar
  \kF \otimes \ShHom_X(\kF, \kG) \lar \kG
  \]
  is an isomorphism on all fibres $X_s$. Because $\kG$ is $S$-flat, this
  is sufficient to conclude that $\kF \otimes p^* \kL \lar \kG$ is an
  isomorphism.
\end{proof}

We shall see below that the sheaf property will follow if we can show that
there exists a certain twisted line bundle on $T$.
For its construction we denote
\[
\widetilde{\Mf}^{(n,d)}_{X/S}(T) =
\left\{
\kF \in \Coh(X_T)\left|
\begin{array}{l}
\kF \text{ is } T\text{-flat, } \kF_{t} \text{ is simple}\\
\rk(\kF_{t})=n, \chi(\kF_{t})=d \quad \forall t \in T
\end{array}
\right.\right\}
\subset\Coh(X_{T})\,,
\]
so that $\underline{\Mf}^{(n,d)}_{X/S}(T) =
\widetilde{\Mf}^{(n,d)}_{X/S}(T)/\sim\;$.
We use functors
$\mathbb{L}_{T}:\widetilde{\Mf}^{(n,d)}_{X/S}(T)\rightarrow\Coh(T)$ which are
defined as follows.
Fix an integer $m\ge m_{0}$ (see Theorem \ref{T:vanishing}) and let
$a,b$ be integers such that $a(d+mn)+b(d+(m+1)n)=1$. Such
integers exist if $\gcd(n,d)=1$. For any $T\in\Ans_{S}$ and
$\kF\in\widetilde{\Mf}^{(n,d)}_{X/S}(T)$ we define
\[\mathbb{L}_{T}(\kF) = \bigl(\det\left(q_{\ast}\kF(m)\right)\bigr)^{\otimes a}
\otimes \bigl(\det\left(q_{\ast}\kF(m+1)\right)\bigr)^{\otimes b}\,.\]

\begin{lemma}\label{L:functor}
  The functors
  $\mathbb{L}_{T}:\widetilde{\Mf}^{(n,d)}_{X/S}(T)\rightarrow\Coh(T)$ have the
  following properties.
  \begin{itemize}
  \item[(a)] If $\kF\in\widetilde{\Mf}^{(n,d)}_{X/S}(T)$, then
    $\mathbb{L}_{T}(\kF)\in\Pic(T)$.
  \item[(b)] For $\lambda\in H^{0}(\kO_{T})$ and
    $\kF\in\widetilde{\Mf}^{(n,d)}_{X/S}(T)$ we have
    $\mathbb{L}_{T}(q^{\ast}(\lambda)\cdot\id_{\kF})
    = \lambda\cdot\id_{\mathbb{L}_{T}(\kF)}$.
  \item[(c)] If $u:T'\subset T$ is an open subset and $v:X_{T'}\subset X_{T}$,
    then $u^{\ast}\circ\mathbb{L}_{T} \cong \mathbb{L}_{T'}\circ v^{\ast}$.
  \end{itemize}
\end{lemma}

\begin{proof}
  (a) By Theorem \ref{T:vanishing} we have
  $H^{1}(\kF(m)_{t})=H^{1}(\kF(m+1)_{t})=0$ for all $t\in T$.
  From \cite[III.3.9]{BS} we obtain  that $q_{\ast}\kF(m)$
  and $q_{\ast}\kF(m+1)$ are locally free, and the result follows.
  Moreover, it follows that the ranks of these locally free sheaves are equal
  to $\chi(\kF(m)_{t})=d+mn$ and $\chi(\kF(m+1)_{t})=d+(m+1)n$, respectively.
  This will be used in the next part of the proof.

  (b) The functor $q_{\ast}$ sends $q^{\ast}(\lambda)\cdot\id_{\kF(m)}$ to
  $\lambda\cdot\id_{q_{\ast}\kF(m)}$ which is sent to $\lambda^{d+mn}\cdot\id$ by
  the functor $\det$.
  Therefore, $\mathbb{L}_{T}(q^{\ast}(\lambda)\cdot\id_{\kF})
    = \lambda^{a(d+mn)+b(d+(m+1)n)}\cdot\id_{\mathbb{L}_{T}(\kF)}
    = \lambda\cdot\id_{\mathbb{L}_{T}(\kF)}$.

  (c) This is clear because tensor product and taking the determinant is
  compatible with pull-back and, if $q'$ denotes the restriction of $q$ to
  $X_{T'}$,  $u^{\ast}\circ q_{\ast} = q'_{\ast}\circ v^{\ast}$ holds.
\end{proof}

\begin{proposition}\label{P:sheaf}
  The functor $\underline{\Mf}^{(n,d)}_{X/S}$ is a sheaf in the analytic
  topology.
\end{proposition}

\begin{proof}
  In order to prove that the separated functor $\underline{\Mf}^{(n,d)}_{X/S}$
  is a sheaf it is sufficient to prove the following.
  For each $T\in\Ans_{S}$, for each open cover $\{T_{i}\}$ of $T$ and any
  $\kF_{i}\in \underline{\Mf}^{(n,d)}_{X/S}(T_{i})$ which satisfy
  $\kF_{j}|_{X_{ij}} \cong \kF_{i}|_{X_{ij}}$ for all
  $i,j$, there exists
  $\kF\in\underline{\Mf}^{(n,d)}_{X/S}(T)$ such that
  $\kF|_{X_{i}} \sim \kF_{i}$. Here we abbreviate
  $X_{i}:=X_{T_{i}}$ and $X_{ij} :=(T_{i}\cap T_{j})\times_{S}X$. More
  explicitly, the equivalence $\kF|_{X_{i}} \sim \kF_{i}$
  means that there exist line bundles $\kL_{i}$ on $T_{i}$ such that
  $\kF|_{X_{i}}\cong \kF_{i} \otimes  q^{\ast}\kL_{i}$.

  To prove this, we denote by $\varphi_{ij}:\kF_{j}|_{X_{ij}}
  \rightarrow \kF_{i}|_{X_{ij}}$ isomorphisms such that
  $\varphi_{ii}=\id$ and $\varphi_{ij}=\varphi_{ji}^{-1}$. On the triple overlaps
  $X_{ijk} = (T_{i}\cap T_{j}\cap T_{k})\times_{S}X$, the composition
  $\varphi_{ij}\circ\varphi_{jk}\circ\varphi_{ki}$ is an automorphism of
  $\kF_{i}|_{X_{ijk}}$. Because $\kF_{i}$ is a simple vector
  bundle on each fibre of $q$, this composition must be a multiple of the
  identity.
  More precisely, by Lemma \ref{L:simple} there exists $\alpha_{ijk} \in
  \kO(T_{i}\cap T_{j}\cap T_{k})$, such that
  $\varphi_{ij}\circ\varphi_{jk}\circ\varphi_{ki}
  = q^{\ast}(\alpha_{ijk})\cdot\id$.
  These functions $\alpha_{ijk}$ define a \v{C}ech cocycle $\alpha$ with values
  in the sheaf of abelian groups $\kO_{T}^{\ast}$.
  Let us denote its cohomology class by $\overline{\alpha}\in
  H^{2}(\kO^{\ast}_{T})$.

  The given sheaves $\kF_{i}$ glue to a sheaf $\kF$ on $X_{T}$
  if and only if $\alpha_{ijk}=1$ for all $i,j,k$.
  Moreover, there exist line bundles $\kL_{i}$ on $T_{i}$ such that the
  $\kF_{i}\otimes q^{\ast}\kL_{i}$ glue to a sheaf $\kF$
  on $X_{T}$ if and only if $\overline{\alpha}=1$ in
  $H^{2}(\kO^{\ast}_{T})$.

  It is convenient to use the language of twisted sheaves in this
  situation.
  The sheaves $\kF_{i}$ together with the isomorphisms
  $\varphi_{ij}$ define a $q^{\ast}\alpha$-twisted sheaf on $X_{T}$. We are
  going to construct an $\alpha$-twisted line bundle on $T$. The existence of it
  implies already $\overline{\alpha}=1$. More explicitly, we can form
  the tensor product of the given $q^{\ast}\alpha$-twisted sheaf with the pull
  back
  of the dual of this $\alpha$-twisted line bundle. This produces the desired
  (untwisted) sheaf $\kF$ on $X_{T}$.

  Let $\kL_{i}=\mathbb{L}_{T_{i}}(\kF_{i})$ and
  $\psi_{ij}=\mathbb{L}_{T_{ij}}(\varphi_{ij})$. Lemma \ref{L:functor} (a)
  implies that $\kL_{i}$ is a line bundle on $T_{i}$.
  Using Lemma \ref{L:functor} (b) and (c) we obtain
  \begin{align*}
    \psi_{ij}|_{T_{ijk}}\circ\psi_{jk}|_{T_{ijk}}\circ\psi_{ki}|_{T_{ijk}}
    &= \mathbb{L}_{T_{ij}}(\varphi_{ij})|_{T_{ijk}}\circ
    \mathbb{L}_{T_{jk}}(\varphi_{jk})|_{T_{ijk}}\circ
    \mathbb{L}_{T_{ki}}(\varphi_{ki})|_{T_{ijk}}\\
    &= \mathbb{L}_{T_{ijk}}(\varphi_{ij}|_{X_{ijk}})\circ
    \mathbb{L}_{T_{ijk}}(\varphi_{jk}|_{X_{ijk}})\circ
    \mathbb{L}_{T_{ijk}}(\varphi_{ki}|_{X_{ijk}})\\
    &= \mathbb{L}_{T_{ijk}}(\varphi_{ij}\varphi_{jk}\varphi_{ki}|_{X_{ijk}})\\
    &= \mathbb{L}_{T_{ijk}}
    \bigl(q^{\ast}(\alpha_{ijk})\cdot\id_{\kF_{i}|_{X_{ijk}}}\bigr)\\
    &= \alpha_{ijk}\cdot\id_{\kL_{i}|_{T_{ijk}}}\;.
  \end{align*}
  This shows that $\bigl(\{\kL_{i}\},\{\psi_{ij}\}\bigr)$ is an
  $\alpha$-twisted line bundle on $T$. This implies $\overline{\alpha}=1$
  which concludes the proof.
\end{proof}

\begin{proof}[Proof of Theorem \ref{T:main}]
  Because we assume $n$ and $d$ coprime, the functors $\mathbb{L}_{T}$
  exist. By Proposition \ref{P:sheaf}, the functor
  $\underline{\Mf}^{(n,d)}_{X/S}$ is a sheaf in the analytic topology.
  Theorem 6.4 and Remark 6.7 from \cite{KO} imply now the claim.
\end{proof}

\begin{remark}\label{R:bounded}
  If we fix $\kF\in\underline{\Mf}^{(n,d)}_{X/S}(T)$, we may define
  \[T_{m} = \{t\in T\mid H^{1}(\kF(m)_{t}) = 0\}\]
  for each integer $m$, i.e.\ $T\setminus T_{m}$ is the support
  of the coherent sheaf $R^{1}q_{\ast}(\kF(m))$.
  Because the fibres $X_{t}$ are curves, $T_{m}$ is the set over which $\kF$ is
  $(m+1)$-regular. The sets $T_{m}$ are open subsets of $T$ such that
  $T_{m}\subset T_{m+1}$ and $T=\bigcup_{m\in\mathbb{Z}} T_{m}$.
  Theorem \ref{T:vanishing} implies that there is an integer $m_{0}$, neither
  depending on $\kF$ nor on $T$, such that $T_{m_{0}} = T$.

  We can look at the open subfunctors
  $\underline{\Mf}^{(n,d,m)}_{X/S} \subset \underline{\Mf}^{(n,d)}_{X/S}$
  defined by the extra condition that $H^{1}(\kF(m)_{t}) = 0$ for all $t\in T$,
  i.e.\ $R^{1}q_{\ast}\kF(m)=0$. Our argument in the proof of
  Proposition \ref{P:sheaf} shows that each
  $\underline{\Mf}^{(n,d,m)}_{X/S}$ is a sheaf. Theorem \ref{T:vanishing}
  implies that $\underline{\Mf}^{(n,d,m)}_{X/S} = \underline{\Mf}^{(n,d)}_{X/S}$
  for all $m\ge m_{0}$.
\end{remark}

\begin{remark}\label{R:alternative}
  A proof of Proposition \ref{P:sheaf}, in which Theorem \ref{T:vanishing} is
  not used, could possibly be carried out as follows.
  Similar to the definition used before, we let
  \[\mathbb{L}_{T}(\kF) = \bigl(\det\left(\kA\right)\bigr)^{\otimes a}
  \otimes \bigl(\det\left(\kB\right)\bigr)^{\otimes b}\,,\]
  but this time we use $\kA=q_{\ast}(\kF|_{\Sigma_{T}})$ and
  $\kB=\mathbf{R}q_{\ast}(\kF)$.
  The key points to be verified are now that $\kA$ is locally free (basically
  because the section $\Sigma$ avoids the singularities of the fibres),
  that $\kB$ is a perfect complex and that $\lambda\cdot\id_{\kF}$ acts as
  $\lambda^{\chi(\kF)}\cdot\id$ on $\det(\kB)$.
\end{remark}

\begin{remark}
  Because local freeness is an open property, Theorem
  \ref{T:main} implies that the sub-functor
  $\underline{\breve{\Mf}}^{(n,d)}_{X/S}$ of $\underline{\Mf}^{(n,d)}_{X/S}$,
  which consists of families of vector bundles, is representable by an open
  subspace $\breve{\Mf}^{(n,d)}_{X/S}$ of $\Mf^{(n,d)}_{X/S}$. In particular,
  under the assumptions made in the introduction, the functor
  $\underline{\Pic}^{d}_{X/S} = \underline{\breve{\Mf}}^{(1,d)}_{X/S}$ is
  representable by a complex space $\Pic^{d}_{X/S}$.
  This functor is defined without sheafification; the elements of
  $\underline{\Pic}^{d}_{X/S}(T)$ are equivalence classes of line bundles on
  $X_{T}$.
\end{remark}

\begin{remark}
  If $S$ is a reduced point, $X$ an irreducible projective curve of
  arithmetic genus $1$, and $n>0$ and $d$ coprime integers, there is an
  isomorphism
  \[\Det:\breve{\Mf}^{(n,d)}_{X} \lar \Pic^{d}_{X}\,.\]
  The morphism $\Det$ is determined by the condition
  $\Det^{\ast}\kL \cong \det(\kP)$, where $\kL$ on $\Pic^{d}_{X}$ and $\kP$ on
  $\breve{\Mf}^{(n,d)}_{X}$ are universal sheaves. Note that, because $g=1$,
  for a vector bundle $V$ on $X$ we have $\deg(V)=\chi(V)$, so that indeed
  $\chi(V)=\chi(\det(V))$.

  It was shown by Atiyah \cite[Theorem 7]{At} if $X$ is smooth and for instance
   in \cite[Theorems 5.1.19 and 5.1.34]{BK-AMS} in the singular case,
  that the map $\Det$ is bijective. Because the tensor product with a line
  bundle of degree $d$ defines an isomorphism between $\Pic^{0}_{X}$ and
  $\Pic^{d}_{X}$ and because $\Pic^{0}_{X}$ is a group, $\Pic^{d}_{X}$ is
  smooth. This is sufficient to conclude from the bijectivity of $\Det$ that
  this morphism in fact is an isomorphism, see e.g.\
  \cite[Theorem 4.6.8]{Taylor}.
\end{remark}

\noindent
\textbf{Acknowledgements}. Both authors would like to thank Manfred Lehn for
discussing the proofs of Lemmas \ref{L:simple} and \ref{L:wellknown} and
Norbert Hoffmann for suggesting the construction in
Remark \ref{R:alternative}.
The work of the first-named author was supported by the DFG grant Bu-1866/2-1.

\end{document}